\documentclass[12pt]{amsart}
\usepackage{amsmath, amssymb, amsthm}
\theoremstyle{plain}

\newtheorem{theorem}{Theorem}[section]
\newtheorem{lemma}[theorem]{Lemma}
\newtheorem{corollary}[theorem]{Corollary}
\newtheorem{proposition}[theorem]{Proposition}
\theoremstyle{definition}
\newtheorem{definition}[theorem]{Definition}
\newtheorem{remark}[theorem]{Remark}
\newtheorem{example}[theorem]{Example}
\newtheorem{application}[theorem]{Application}

\begin{document}

\title[Classification of constant angle hyp. in warped products]{Classification of constant angle 
hypersurfaces in warped products via eikonal functions}

\author{Eugenio Garnica}       
\author{Oscar Palmas}
\author{Gabriel Ruiz-Hern\'andez}


\date{} %leave empty

\maketitle
%%%%%%%%%%%%%%%%%%%%%%%%%%%%%%%%%%%%%%%%%

\begin{abstract}
Given a warped product of the real line with a Riemannian manifold of arbitrary dimension, we classify the hypersurfaces whose tangent spaces make
a constant angle with the vector field tangent to the real direction. We show that this is a natural setting in which to extend previous results in this direction made by several authors. Moreover, when the constant angle hypersurface is a graph over the Riemannian manifold, we show that the function involved satisfies a generalized eikonal equation, which we solve via a geometric method. 
In the final part of this paper we prove that minimal constant angle hypersurfaces are cylinders over minimal submanifolds.
\end{abstract}

\section*{Introduction}

Several classical, well-known geometric objects are defined in terms of making a constant angle with a given, distinguished direction. Firstly, classical helices are curves making a constant angle with a fixed direction. A second example is the logarithmic spiral, the {\em spira mirabilis} studied by Jacob Bernoulli, which makes a constant angle with the radial direction. In a third famous example which had applications to navigation, the {\em loxodromes} or {\em rhumb lines} are those curves in the sphere making a constant angle with the sphere meridians.

Recently, several authors had established and investigated some generalizations of the above situation. In 2007, F. Dillen et al. characterized those surfaces $M$ in $\mathbb S^2\times\mathbb R$ whose normal vector $\xi$ makes a constant angle $\theta$ with the direction tangent to $\mathbb R$ (see \cite{MR2346426}). 
Two years later, F. Dillen and M.~I. Munteanu gave in \cite{MR2496114} a similar characterization theorem for constant angle surfaces in the product $\mathbb H^2\times\mathbb R$, using the hyperboloid model for the hyperbolic plane $\mathbb H^2$. In the final part of the paper they classified the constant angle surfaces with constant mean curvature in this Riemannian product.

Another nice paper in this direction is \cite{MR2681099}, where M.~I.
Munteanu made a review of some applications of constant angle surfaces and gave a complete classification of the so-called constant slope surfaces in $\mathbb{R}^3$, that is, those surfaces making a constant
angle with the radial position vector field. He showed that a surface $S\subset\mathbb R^3$ is a constant slope surface iff either it is an Euclidean $2$-sphere centered at the origin or it can be parameterized by
\[
r(u,v)=u\sin\theta(\cos\xi f(v)+\sin \xi f(v)\times f'(v)),
\]
where $\theta$ is a constant different from $0$, $\xi=\xi(u)=\cot\theta\log u$ and $f$ is a unit speed curve on the Euclidean sphere $\mathbb S^2$.

It is also worth mentioning the recent paper \cite{DMVV}, where Dillen,         
Munteanu, Van der Veken and Vrancken classified the constant angle surfaces in the 
warped product $I\times_\rho\mathbb R^2$. We will discuss the relation of this and other works with ours in Section \ref{sec:examples}.
This class of surfaces or curves making a constant angle with respect to some direction have been also investigated
in Minkowski space, see \cite{MR2778006} and \cite{LoMu} for details.

Using another approach, A. Di~Scala and the third named
author studied in \cite{MR2520724} the helix submanifolds of Euclidean spaces, i.~e., submanifolds making
a constant angle with a constant direction. They builded constant angle hypersurfaces of $\mathbb R^{n+1}$, as follows: Given an orientable hypersurface $L$ of $\mathbb R^n$ with a unit normal vector field $\eta$, let $r:L\times\mathbb R\to\mathbb R^{n+1}$ be defined by
\[
r(x,s)=x+s\left((\sin\theta)\eta(x)+(\cos\theta)d\right),
\]
where $\theta$ is constant and $d=(0,\dots,0,1)$. Then $f$ parameterize a hypersurface making a constant angle $\theta$ with the fixed direction determined by $d$. Moreover, they showed that, except for some trivial cases, any helix hypersurface admits locally such a parametrization. They also showed that these non-trivial constant angle submanifolds are given locally as graphs of functions whose gradient has constant length (that is, solutions of the so-called eikonal equation). In \cite{MR2681534}, they showed further that any function satisfying the eikonal equation may be characterized as a distance function relative to an embedded hypersurface in the ambient space.

All of the above results suggest the existence of a general framework in which it is natural to consider the study of constant angle submanifolds. As it will turn out along this paper, a natural choice for that purpose is an ambient space $\bar M$ given as a warped product of the form $I \times_\rho \mathbb{P}^n$, where $I$ is an open interval and $\rho:I\to\mathbb R^+$ is a smooth positive function. We consider those submanifolds making a constant angle with the vector field $\partial_t$ tangent to the $\mathbb R$-direction. Of course, the case of the Euclidean ambient
space is obtained by considering $\mathbb P^n=\mathbb R^n$ and the
constant warping function $\rho\equiv 1$.

The plan of this paper is the following. Section \ref{sec:preliminares} gives the basic geometric properties of constant angle hypersurfaces in a warped product, showing that they have a rich extrinsic and intrinsic geometry. In Theorem \ref{teo:extrinsic-geometryofCAH} we prove that if the projection of $\partial_t$ to the tangent space of a constant angle hypersurface does not vanish, it determines a principal direction on the hypersurface. In the terminology of the recent works \cite{MR2506241}, \cite{Di-Mu-Ni} and \cite{Mu-Ni-II}, the hypersurface has a \emph{canonical principal direction} relative to the distinguished vector field $\partial_t$. Also, we prove that the integral lines of this tangential component are lines of curvature and geodesics of the hypersurface.

In Section \ref{sec:main} we state our main result giving a complete characterization of constant angle hypersurfaces in $I \times_\rho
\mathbb{P}^n$ (see Theorem \ref{graph-of-transnormal}):

\medskip

\textit{Let} $\overline{M}^{n+1}$ \textit{be the warped product} $I\times_\rho\mathbb
P^n$. \textit{A connected hypersurface} $M$ \textit{of} $\overline{M}$ \textit{is a constant angle hypersurface in }$\bar M$ \textit{if and only if it is an open subset of  either }
\begin{itemize}
\item \textit{A cylinder of the form} $I\times L^{n-1}$, \textit{where} $L$ \textit{is a hypersurface of} $\mathbb P$; \textit{or}

\item \textit{The graph of a function} $f:\mathbb P\to\mathbb R$ \textit{satisfying the} generalized eikonal equation \begin{equation}\label{eq:eikonal}
\vert\nabla f\vert=C\cdot(\rho\circ f),
\end{equation}
\textit{where }$C$ \textit{is a constant, }$\rho$ \textit{stands for the warping function and the graph of} $f$ \textit{is defined as the set
of points }$(f(p),p)$ \textit{with} $p\in \mathbb P$.
\end{itemize}

\bigskip

We also give a geometric method to build the solutions of the generalized eikonal equation, by generalizing the technique given in \cite{MR2681534} for the case of the classical eikonal equation. Our result in the context of constant angle hypersurfaces is the following (see Corollary \ref{cor:construccion}):

\medskip

\textit{Let} $\overline{M}^{n+1}$ \textit{be the warped product}
$I\times_\rho\mathbb P^n$. \textit{A connected hypersurface in} $\bar M$ \textit{is a constant angle hypersurface with} $\theta\in(0,\pi/2)$ \textit{if and only if it is the graph of a function} $f:\mathbb P\to\mathbb R$ \textit{of the form} $f=h\circ d$, \textit{where} $d$ \textit{measures the distance to a fixed orientable hypersurface} $L\subset\mathbb P$ \textit{and} $h$ \textit{satisfies}
\[
h^{-1}(s)=\int_{s_0}^s\frac{d\sigma}{C\rho(\sigma)},
\]
\textit{with} $C=\tan\theta$.

\medskip

%In Propositions \ref{prop:solution-eikonal} and \ref{coro:local-solutions}, we give this method and show the
%existence and uniqueness of (local) solutions to this equation.

In Section \ref{sec:examples} we show the relation between the parametrizations of constant angle surfaces obtained by the authors already mentioned in this Introduction and our language. Note that our setting includes all codimension~$1$ cases, and in particular, the case of surfaces in every $3$-dimensional warped product of the form $I\times_\rho\mathbb P^2$.

Finally, in Section \ref{sec:minimas} we prove that minimal constant angle hypersurfaces are cylinders over a minimal submanifold
of codimension two. We deduce this result from the following nice property:\\ 
\textit{Let $f: \Omega \subset \mathbb{R}^n \longrightarrow \mathbb{R}$ 
be a smooth function with connected open domain $\Omega$. If $f$ is harmonic and eikonal then $f$ is linear in $\Omega$}.

\section{A canonical principal direction}\label{sec:preliminares}

Throughout this paper, we will use the following notations:
\begin{itemize}
\item $\overline{M}^{n+1}$ will denote a warped product of the form $I \times_\rho \mathbb{P}^n$, where $I$ is an open interval, $\mathbb P$ is a Riemannian manifold and $\rho:I\to\mathbb R^+$.

\item $\overline{\nabla}$ is the Riemannian connection on $\overline{M}$ relative to the warped product metric.

\item $\partial_t$ will denote the unit vector field tangent to the $\mathbb R$-direction in $\overline{M}$.

\item $M$ will be a connected orientable hypersurface in $\overline{M}$.

\item $\nabla$ will denote the induced Riemannian connection on $M$.

\item $\xi\in\mathfrak{X}(M)$ will be a unit vector field, everywhere normal to $M$.

\item $\theta$ will denote the function on $M$ measuring the angle between $\partial_t$ and $\xi$.
\end{itemize}

\begin{definition}
We say that $M$ is a {\em constant angle hypersurface} iff the angle function $\theta$ is constant along $M$.
\end{definition}

\begin{remark} Given a constant angle hypersurface, we may choose the orientation of $M$ so that $\theta\in[0,\pi/2]$, as we will do.
\end{remark}

Our aim here is to classify all constant angle hypersurfaces $M$ of the warped product $I\times_\rho\mathbb P^n$. A trivial case occurs when $\theta\equiv 0$. In the language of the warped product structure, $\xi=\partial_t$ and then a connected constant angle hypersurface is contained in a slice $\{t_0\}\times\mathbb P$. So, we suppose in this section that $\theta\in(0,\pi/2]$.

\medskip

Let us fix some additional notation. As usual, we have the Gauss and Weingarten equations for hypersurfaces:
\[
\overline{\nabla}_YZ=\nabla_YZ+II(Y,Z),\quad
\overline{\nabla}_Y\xi=-A_\xi Y,
\]
where $Y,Z\in\mathfrak{X}(M)$, $II$ is the second fundamental form
of $M$ and $A_\xi$ is the shape operator associated to $\xi$. Recall also that
$II$ and $A_\xi$ are related by the formula
\[
\langle II(Y,Z),\xi \rangle = \langle A_\xi Y, Z \rangle.
\]

Let $\partial_t^\top$ be the component of $\partial_t$ tangent to $M$, that is,
\[
\partial_t^\top=\partial_t-\langle \partial_t,\xi\rangle\xi, \]

Note that $\theta\in(0,\pi/2]$ implies $\partial_t^\top\neq 0$ and we may
define
\begin{equation}\label{def:T}
T=\frac{\partial_t^\top}{\vert \partial_t^\top \vert}.
\end{equation}

Hence we may write
\begin{equation}\label{eq:descomposicion}
\partial_t=(\sin\theta)T+(\cos\theta)\xi.
\end{equation}

Now we are ready to give some basic geometric properties of the constant angle hypersurfaces.

\begin{theorem}
\label{teo:extrinsic-geometryofCAH} Let $M$ be a constant angle
hypersurface of $\overline{M}^{n+1}$ such that $\theta\in(0,\pi/2]$. Then the integral lines of
the vector field $T$ defined in $(\ref{def:T})$ are
lines of curvature of $M;$ in fact,
\[ A_\xi T= - \cos\theta\frac{\rho'}{\rho}T. \]

In other words, $T$ is a principal direction of $M$. Moreover, these lines are geodesics of $M$, that is, $ \nabla_T T =0$.

Additionally, the integral lines of $T$ are geodesics of
$\overline{M}$ iff either $\partial_t$ is parallel or $\theta=\pi/2$.
\end{theorem}

\begin{proof}

Suppose first that $\theta\in(0,\pi/2)$, which implies $\cos\theta\ne 0$. Differentiating (\ref{eq:descomposicion}) with respect to a vector field $W\in\mathfrak{X}(M)$, we obtain
\begin{equation}\label{eq:dwt}
\overline{\nabla}_W \partial_t= (\sin\theta)\overline{\nabla}_WT+(\cos\theta) \overline{\nabla}_W\xi.\end{equation}

Suppose additionally that $ \langle W,T \rangle =0$ or, equivalently, $ \langle W,\partial_t \rangle =0$. To calculate $\overline\nabla_W\partial_t$, we may suppose that $W$ is given as a lifting of a vector field on $\mathbb P$ and use standard derivation formulas in warped products (see \cite{MR719023}, p. 296, Prop. 35, for example) to obtain that $\overline{\nabla}_W \partial_t= (\rho'/\rho)W$. Taking the components tangent and normal to $M$ in the above formula and using that $\cos\theta\ne 0$, we have
\[
A_\xi W=-\frac{\rho'}{\rho \cos\theta} W + (\tan\theta) \nabla_W T,
\]
and $II(W,T)= 0$, which implies that
\[
\langle A_\xi T, W \rangle = \langle A_\xi W,T \rangle = \langle II(W,T) , \xi \rangle=0
\]
for every $W\in\mathfrak{X}(M)$ such that $ \langle W,T \rangle =0$. In turn, this fact implies that $A_\xi T$ is a scalar multiple of $T$, i.e., $T$ is a principal
direction of $M$.

%
%Using (\ref{eq:conformal}) again it is easy to see that $W(|X|)=0$ to obtain
%%\[
%%\overline{\nabla}_W\left(\frac{X}{\vert X \vert}\right) =
%%\frac{\varphi}{\vert X \vert}W.
%%\]
%%
%%Thus, in this case we obtain
%\[
%A_\xi W=-\frac{\varphi}{|X| \cos\theta} W + (\tan\theta) \nabla_W T \ \ { \rm and } \ \
%II(W,T)=0.
%\]
%

We return to the general expression (\ref{eq:dwt}) and take $W=T$. In order to use the derivation formulas for warped products again, we write
\[
T=(\sin\theta)\partial_t+(\cos\theta)\left[(\cos\theta)T-(\sin\theta)\xi\right],
\]
and note that the vector field $(\cos\theta)T-(\sin\theta)\xi$ is orthogonal to $\partial_t$. Hence,
\begin{eqnarray*}
\overline\nabla_T\partial_t & = & (\sin\theta)\overline\nabla_{\partial_t}\partial_t
+(\cos\theta)\overline\nabla_{\left[(\cos\theta)T-(\sin\theta)\xi\right]}\partial_t \\
 & = & (\cos\theta)\frac{\rho'}{\rho}\left[(\cos\theta)T-(\sin\theta)\xi\right];
\end{eqnarray*}
so that the tangent and normal components of (\ref{eq:dwt}) are
\[
\cos^2\theta\frac{\rho'}{\rho}T=(\sin\theta)\nabla_TT-(\cos\theta)A_\xi T
\]
and
\[
-\sin\theta\cos\theta\frac{\rho'}{\rho}\xi=(\sin\theta)II(T,T).
\]

From the first of these expressions, since $A_\xi T$ is a scalar multiple of $T$ (and $\sin\theta\ne 0$), we deduce that the same happens with $\nabla_TT$; but as $T$ is a unit vector field, we have $\nabla_T T=0$; i.e., the integral lines of $T$ are geodesics in $M$. Also,
\[A_\xi T = - \cos\theta\frac{\rho'}{\rho} T,\]
meaning that $T$ is a principal direction. In the case of the second fundamental form, we have
\[II(T,T)= - \cos\theta\frac{ \rho' }{\rho}  \xi.\]

Since we are analyzing the case $\cos\theta\ne 0$,  $II(T,T) = 0$ if and only if $\rho'=0$; i.e., $\rho$ is constant. In this case, $\overline\nabla_W\partial_t=0$ for every vector field $W\in\mathfrak{X}(\overline{M})$. That is, the integral lines of $T$ are geodesics of $\overline{M}$ if and only if $\partial_t$ is parallel.

The analysis in the case $\theta=\pi/2$ is similar, but easier, since in this case equation (\ref{eq:descomposicion}) reduces to $T=\partial_t$. We have that $\overline\nabla_TT=\overline\nabla_{\partial_t}\partial_t=0$ and then the integral lines of $T$ are geodesics of $\bar M$, thus they also are geodesics of $M$. If $W\in\mathfrak{X}(M)$ is orthogonal to $T$ we have on one hand
\[
\overline\nabla_WT=\overline\nabla_W \partial_t=\frac{\rho'}{\rho}W,
\]
and on the other hand, $\overline\nabla_WT=\nabla_WT+II(W,T)$, which implies that $II(W,T)=0$. As in the previous case, this in turn implies that $A_\xi T$ is a scalar multiple of $T$ and $T$ is a principal direction. In fact, since $\overline\nabla_TT=\nabla_TT+II(T,T)=0$, we have
\[
\langle A_\xi T,T\rangle =\langle II(T,T),\xi\rangle =0,
\]
and then $A_\xi T=0$ and $T$ is a principal direction.
\end{proof}

Theorem \ref{teo:extrinsic-geometryofCAH} says that the constant
angle hypersurfaces with $\theta\in(0,\pi/2]$ are examples of hypersurfaces with a canonical principal direction, which means that there exists a vector field in the ambient such that the
component of this vector field tangent to the surface is a principal direction for the shape operator of the surface. This notion has been studied recently by several
authors; see, for example \cite{MR2506241}, \cite{Di-Mu-Ni} and \cite{Mu-Ni-II}, where the authors classify surfaces with a canonical principal direction in $\mathbb H^2\times\mathbb R$, $\mathbb S^2\times\mathbb R$ and $\mathbb R^2\times\mathbb R$, respectively.

\section{Construction and characterization of constant angle hypersurfaces}\label{sec:main}

In this section we prove our main results, classifying the constant angle hypersurfaces in any warped product of the form $I\times_\rho\mathbb P^n$. First we consider the case of $\theta=\pi/2$:

\begin{proposition}\label{prop:pisobre2}
Let $M$ be a connected hypersurface of $I\times_\rho\mathbb P^n$. $M$ is a constant angle hypersurface with $\theta=\pi/2$ if and only if $M$ is an open subset of a cylinder $I\times L^{n-1}$, where $L$ is a $(n-1)$-dimensional hypersurface of $\mathbb P$.
\end{proposition}

\begin{proof}
Suppose $M$ is a constant angle hypersurface with $\theta=\pi/2$. By transversality, the intersection of $M$ with a fixed slice $\{t_0\}\times\mathbb P^n$ is (isometric to) a hypersurface $L$ of $\mathbb P^n$. Since $\partial_t$ is everywhere tangent to $M$ in this case, we reconstruct $M$ by departing from this intersection and following the flow of $\partial_t$, obtaining the aforementioned cylinder. The converse is clear.
\end{proof}

In view of this result, we may suppose from now on that $\theta\in[0,\pi/2)$. Using transversality, we may suppose additionally that $M$ is given locally as a graph of a real function $f:\mathbb P\to I$. We will prove that such a graph is a constant angle hypersurface if and only if $f$ satisfies a condition on the norm of its gradient (see equation (\ref{eq:eikonal})). In the following definition we fix the classical terminology for this kind of functions.

\begin{definition} \label{defeikonal}
Let $\mathbb P^n$ be a Riemannian manifold and $f:\mathbb P
\rightarrow I$ a differentiable function, where $I$ is a real interval. We say that $f$
is {\em eikonal} if it is a solution of the {\em eikonal equation}
\[ | \nabla f |=C,\] where $\nabla f$ denotes the gradient of $f$
and $C$ is a given constant. More generally, let $\rho:I\to\mathbb R^+$ be a differentiable positive function. We say
that $f$ is a {\em transnormal function} if it satisfies the
generalized eikonal equation (\ref{eq:eikonal}), namely,
\[
\vert\nabla f\vert=C\cdot(\rho\circ f).
\]\end{definition}

The concept of transnormal function is related to the class of
submanifolds called isoparametric submanifolds which are level
hypersurfaces of isoparametric functions. According to
\cite{MR0901710}, a transnormal function
 is a smooth function $f$
satisfiying the equation $|\nabla f|^2=b \circ f$, where $b$ is a
smooth function which can be zero at some points. In our case
$b=C\rho>0$. An isoparametric function is a transnormal function
that also satisfies the condition $\Delta f = a \circ f$, where
$a$ is a smooth function. It is well known that Cartan investigated such functions on
space forms; see \cite{MR1990032} and \cite{MR0901710} for more
details. An interesting result in \cite{MR0901710}, is that a
transnormal function in $\mathbb{S}^n$ or in $\mathbb{R}^n$ is
isoparametric.

\medskip

The next theorem is our main result, giving the precise relation between the transnormal functions and the constant angle
hypersurfaces.

\begin{theorem}
\label{graph-of-transnormal}
Let $\overline{M}^{n+1}$ be the warped product $I\times_\rho\mathbb
P^n$. A connected hypersurface $M$ of $\overline{M}$ is a constant angle hypersurface in $\bar M$ if and only if it is an open subset of  either
\begin{itemize}
\item A cylinder of the form $I\times L^{n-1}$, where $L$ is a hypersurface of $\mathbb P$; or

\item The graph of a transnormal function $f:\mathbb P\to I$ satisfying equation (\ref{eq:eikonal}) for the warping function $\rho$. Here the graph of $f$ is defined as the set
of points $(f(p),p)$ with $p\in \mathbb P$.
\end{itemize}
\end{theorem}
\begin{proof} Let $M$ be a constant angle hypersurface in $\overline{M}$. By Proposition \ref{prop:pisobre2}, we may suppose that $\theta\in[0,\pi/2)$ and that $M$ is a graph of a function $f$.
Let us denote by $\nabla
f$ the lift to $\overline{M}$ of the gradient of $f$. Then it is
easy to see that a vector field $\xi$ everywhere normal to the
graph of $f$ may be chosen as
\[
\xi=(\rho\circ f)^2\partial_t-\nabla f.
\]

Using the definition of the warped product metric and the fact
that $\partial_t$ and $\nabla f$ are orthogonal, we have that the
square of the norm of $\xi$ is given by
\[
\langle\xi ,\xi\rangle=(\rho\circ f)^4+(\rho\circ
f)^2\,\vert\nabla f\vert^2=(\rho\circ f)^2((\rho\circ
f)^2+\vert\nabla f\vert^2),
\]
and consequently the angle $\theta$ between $\xi$ and $\partial_t$ satisfies

\[
\label{dfn:angle-of-graphs}
\cos\theta= \left\langle \frac{\xi}{\vert\xi\vert},
\partial_t\right\rangle  =\frac{\rho\circ f}{\sqrt{(\rho\circ f)^2+\vert\nabla f\vert^2}}.
\]
Note that $\cos\theta\ne 0$ for $\theta\in[0,\pi/2)$. Hence we may express $\vert\nabla f\vert$ in terms of $\rho\circ f$ as
$$ \vert\nabla f\vert = (\tan \theta)( \rho\circ f ),$$
which means that $f$ is transnormal with $C=\tan\theta$.

Conversely, if we consider the graph of a transnormal function satisfying equation (\ref{eq:eikonal}), the angle $\theta$ between its normal $\xi$ and $\partial_t$ is such that
\begin{equation}
\label{dfn:angle-of-CAgraphs}
\cos\theta=\left\langle \frac{\xi}{\vert\xi\vert},
\partial_t\right\rangle  = \frac{\rho\circ f}{\sqrt{(\rho\circ f)^2+\vert\nabla f\vert^2}}=\frac{1}{\sqrt{1+C^2}};
\end{equation}
meaning that the graph of $f$ is a constant angle
hypersurface.
\end{proof}

In short, Theorem \ref{graph-of-transnormal} proves that every constant angle hypersurface
is locally the graph of a function satisfying a partial differential equation on a Riemannian manifold $\mathbb{P}^n$, the generalized eikonal equation (\ref{eq:eikonal}). In the final part of this section we will solve this equation explicitly by a geometric method using the distance function to an arbitrary hypersurface in $\mathbb{P}^n$.

\medskip

As a first step, in our next Proposition we prove the (local) existence of solutions using a constructive method.

\begin{proposition}
\label{prop:solution-eikonal} Let $\mathbb P^n$ be a Riemannian
manifold and $\rho:I \to\mathbb R^+$ a differentiable
positive function. Fix an orientable hypersurface $L\subset\mathbb
P$ and a tubular neighborhood $L_\epsilon$ of $L$ such that the
distance function $d$ to $L$ is well-defined in $L_\epsilon$ and is
differentiable in $L_\epsilon\setminus L$. Also, define a real
valued and invertible function $h:I\to\mathbb R^+$ by
\begin{equation}\label{eq:h-inversa}
h^{-1}(s)=\int_{s_0}^s\frac{d\sigma}{C\rho(\sigma)},
\end{equation}
where $C\ne 0$. Then $f=h\circ d$ is transnormal in
$L_\epsilon\setminus L$.
\end{proposition}

\begin{proof} It is well-known that $\vert\nabla d\vert=1$ in $L_\epsilon\setminus L$; then,
\begin{eqnarray*}
\vert\nabla f\vert & = & \vert\nabla(h\circ d)\vert =(h'\circ
d)\vert\nabla d\vert = h'\circ d \\ & = &
\frac{1}{(h^{-1})'(h\circ d)}=C\cdot(\rho\circ h\circ
d)=C\cdot(\rho\circ f),
\end{eqnarray*}
which proves the claim. \end{proof}

Now we analyze the (local) uniqueness of solutions of the generalized eikonal equation. We will use the results proved by Di~Scala and the
third named author in \cite{MR2681534}, where they studied the local uniqueness of the solutions of an
eikonal equation.

\begin{proposition}
\label{coro:local-solutions}
Let $f:\mathbb P\to I$ satisfy $
\vert\nabla f\vert=C\cdot(\rho\circ f)$ for $C\ne 0$.
Then $f$ is given locally as in Proposition
$\ref{prop:solution-eikonal}$.
\end{proposition}
\begin{proof}
Let $d=h^{-1} \circ f$, where $h^{-1}$ is defined in equation
(\ref{eq:h-inversa}). Let us calculate the gradient of $d$ in
$\mathbb{P}$ :
\[ \nabla d = \nabla (h^{-1} \circ f)= ((h^{-1})'\circ f) \nabla f = \frac{1}{C \cdot ( \rho \circ f ) } \nabla
f.\]

Therefore, $|\nabla d| =  1 $. Theorem 5.3 in \cite{MR2681534}
implies then that for every point $p \in \mathbb P$ there exists a
neighborhood $U$ of $p$ in $\mathbb P$ and a hypersurface $L
\subset \mathbb P$ such that $d|_U$ measures the distance from a point in $U$ to the hypersurface
$L$.
This proves that $f=h
\circ d$ has the form given in Proposition
\ref{prop:solution-eikonal}.
\end{proof}

We are ready to translate the above results to our constant angle hypersurfaces setting.

\begin{corollary}\label{cor:construccion}
Let $\overline{M}^{n+1}$ be the warped product
$I\times_\rho\mathbb P^n$. A connected hypersurface in $\bar M$ is a constant angle hypersurface with $\theta\in(0,\pi/2)$ if and only if it is the graph of a function $f:\mathbb P\to\mathbb R$ of the form $f=h\circ d$, where $d$ measures the distance to a fixed orientable hypersurface $L\subset\mathbb P$ and $h$ satisfies
\[
h^{-1}(s)=\int_{s_0}^s\frac{d\sigma}{C\rho(\sigma)}.
\]
with $C=\tan\theta$.
\end{corollary}

\section{Applications and Examples}\label{sec:examples}

In this section we will construct some examples of constant angle hypersurfaces and will show the relation of 
our construction with those made in the papers already mentioned in the Introduction.

\begin{example}
Let us consider the upper-half space model for the hyperbolic
space $\mathbb{H}^{n+1}$, which can be expressed as the warped
product $(0, \infty) \times_\rho \mathbb R^n$, where
$\rho(t)=1/t$. Then, taking $s_0=1$, \[r=h^{-1}(s)
=\int_1^s\frac{d\sigma}{C\rho(\sigma)}= \frac 1 C\int_1^s \sigma
\,d\sigma = \frac{s^2-1}{2C}.\]

Hence, $s=h(r)=\sqrt{2Cr+1}$. The hypersurface we consider is $L=
\mathbb R^{n-1}$, identified as usual with the points
$(x_1,\dots,x_{n-1},0)$ so that the (oriented) distance function
to $L$ is $x_n$, the $n$-th coordinate function on $\mathbb R^n$.

Therefore, the explicit expression of the function $f=h\circ d$ is
\[
f(x_1,\dots,x_n)=h\circ d(x_1,\dots,x_n)=h(x_n)=\sqrt{2Cx_n+1}.
\]

We calculate the gradient of $f$ as
\[
\nabla f(x_1,\dots,x_n)=\frac{C}{\sqrt{{2Cx_n+1}}}\ \partial_n,
\]
where $\partial_n=\partial_{x_n}$. Note that
\[
\vert \nabla f(x_1,\dots,x_n) \vert^2 =
\frac{C^2}{2Cx_n+1}=C^2(\rho\circ f)^2(x_1,\dots,x_n).
\]
\end{example}

\begin{application} In \cite{MR2681099}, Munteanu studied the surfaces
in three-dimensional Euclidean space whose normal vector at a
point makes a constant angle with the position vector of that
point, showing (Theorem 1 in \cite{MR2681099}) that a
\label{teo:munteanu} constant angle
surface is an open part of the Euclidean $2$-sphere or it can be
parameterized by
\begin{equation}\label{eq:munteanu}
r(u,v)=u\big\{\sin\theta[\cos(\cot\theta\ln
u)\,\alpha(v)+\sin(\cot\theta\ln
u)\cdot\alpha(v)\times\alpha'(v)]\big\},
\end{equation}
where $\theta\neq 0$ and $\alpha$ is a unit speed curve
$\alpha:I\to\mathbb S^2$.

To translate Munteanu's analysis to our context, note that the Euclidean $3$-space minus the
origin is isometric to the warped product
\[
(0,\infty)\times_\rho\mathbb
S^2(\sin\theta),\quad\rho(t)=\frac{t}{\sin\theta};
\]
here $\mathbb S^2(\sin\theta)$ denotes a $2$-dimensional sphere
with radius $\sin\theta$. Of course, the natural isometry of this
warped product with $\mathbb R^3\setminus\{0\}$ is given
explicitly by $(t,p)\mapsto tp$.

To be able to compare Munteanu's result with our Corollary
\ref{cor:construccion}, we note that the function $h$ given by
equation (\ref{eq:h-inversa}) is given by
\[
h^{-1}(s)=\int_1^s\frac{d\sigma}{C\rho(\sigma)}=\frac {\sin\theta}
C \ln s.
\]

Also, we will obtain an expression for the distance function in
$\mathbb S^2$ to the curve $\alpha$ that appears in (\ref{eq:munteanu}). Note that the expression in braces in
(\ref{eq:munteanu}) gives a point $\varphi(u,v)$ in $\mathbb
S^2(\sin\theta)$ and that its distance $d=d(\varphi(u,v))$ to
$\alpha(v)$ is precisely the product of the radius and the angle
between the two vectors; i.e.,
\[
d(\varphi(u,v))=\sin\theta\cdot\cot\theta\cdot\ln
u=\cos\theta\cdot\ln u;
\]
recalling that $C$ may be seen as $\tan\theta$, we have
\[
d(\varphi(u,v))=h^{-1}(u),
\]
which gives
\[
f(\varphi(u,v))=h\circ d(\varphi(u,v))=u.
\]

This fact means that a constant angle surface in
$(0,\infty)\times_\rho\mathbb S^2(\sin\theta)$ is given by the
graph $(f(\varphi(u,v)),\varphi(u,v))$ of $f$, i.~e., by
\[ (u,\varphi(u,v)) = (u,\sin\theta[\cos(\cot\theta\ln
u)\,\alpha(v)+\sin(\cot\theta\ln
u)\cdot\alpha(v)\times\alpha'(v)]);
\]
but this expression corresponds precisely to equation (\ref{eq:munteanu})
via the aforementioned isometry of $(0,\infty)\times_\rho\mathbb
S^2(\sin\theta)$ with the Euclidean space. Thus, we recover
Munteanu's result.
\end{application}

\begin{application}
In our last comparison we consider the work \cite{DMVV}, where Dillen et al. analyzed the hypersurfaces in the warped product $I\times_\rho\mathbb R^2$ making a constant angle
with the vector field $\partial_t$. Theorem 1 in \cite{DMVV} states that an isometric immersion
$r : M^2 \to \overline{M}=I \times_\rho\mathbb R^2$ defines a
surface with constant angle $\theta\in [0, \pi/2]$ if and only if,
up to rigid motions of $\overline{M}$, one of the following holds
locally:
\begin{enumerate}
\item\label{item1} There exist parameters
$(u, v)$ of $M$, with respect to which the
immersion $r$ is given by \begin{multline}\label{eq:dillen} r(u, v) =
\left( u \sin\theta, \cot\theta\left(\int^{u \sin\theta}
\frac{d\sigma}{\rho(\sigma)}\right)\cos v-\int^v
g(\sigma)\sin\sigma\,d\sigma,\right.\\
\left.\cot\theta\left(\int^{u \sin\theta}
\frac{d\sigma}{\rho(\sigma)}\right)\sin v+\int^v
g(\sigma)\cos\sigma\,d\sigma\right)\end{multline} for some smooth
function $g$.

\item\label{item2} $r(M)$ is an open part of the cylinder $x-G(t)
= 0$ for the real function $G$ given by \[G(t) =
\cot\theta\int^t\frac{d\sigma}{\rho(\sigma)}.\] $($Here $(x,y)$ are
the standard coordinates in $\mathbb R^2$.$)$

\item\label{item3} $r(M)$ is an open part of the surface $t = t_0$
for some real number $t_0$, and $\theta = 0$.

\end{enumerate}

We will discuss items (\ref{item1}) and (\ref{item2}) of this
theorem. In relation with item (\ref{item2}) and in analogy with
our previous discussion of Munteanu's work, we see that the
function $G$ may be written in our terminology as
\[G(t) =
\cot\theta\int^t\frac{d\sigma}{\rho(\sigma)}=
\int^t\frac{d\sigma}{C\rho(\sigma)}=h^{-1}(t).\]

To obtain the cylinder $x-G(t)=0$, we proceed as follows: We build
a constant angle curve in the $(t,x)$-plane, that is, a curve
making a constant angle with the vertical vector field
$\partial_t$. Note that this plane is a warped product
$I\times_\rho\mathbb R$.

By Corollary \ref{cor:construccion}, we may build this curve by first
taking a codimension one manifold in $\mathbb R$, i.e., fixing a
point in the real axis, which we may take as the origin. Next, we
calculate the distance function $d$ in $\mathbb R$ to this point,
which obviously gives $d(x)=x$. Hence, the graph of $f=h\circ
d=h=G^{-1}$ is the constant angle curve we were looking for. By
taking the cylinder over this curve in the $3$-dimensional space,
we obtain the constant angle surface given in item (\ref{item2}).

To analyze item (\ref{item1}), we define the following curve
$\alpha(v)$ in the $(x,y)$-plane:
\[
\alpha(v)= \left( -\int^v g(\sigma)\sin\sigma\,d\sigma, \int^v
g(\sigma)\cos\sigma\,d\sigma\right);
\]
which may be obtained from the second and third coordinates in (\ref{eq:dillen})
making $u=0$.\\ 
Note that $\alpha'(v)=g(v)(-\sin v,\cos v)$, so that
$(\cos v,\sin v)$ is a unit vector field everywhere normal to this
curve. An easy calculation shows that the second and third coordinates in (\ref{eq:dillen}) give a parametrization $\varphi(u,v)$ of a neighborhood of $\alpha$ by Fermi coordinates; in fact, the distance of a point in this neighborhood to the curve $\alpha$ is precisely
\[
d(\varphi(u,v))=\cot\theta\left(\int^{u \sin\theta}
\frac{d\sigma}{\rho(\sigma)}\right),\]
which is equal to $h^{-1}(u\sin\theta)$ in our terminology. From this we have that the eikonal function $f$ given in Corollary \ref{cor:construccion} is
\[
f(\varphi(u,v))=h\circ d(\varphi(u,v))=u\sin\theta;
\]
that is, equation (\ref{eq:dillen}) is the expression of the graph of $f$ in $I\times_\rho\mathbb R^2$.
\end{application}

\begin{remark}
Note that instead of $u\sin\theta$ we may use a function $\psi(u)$
in the upper limit of the integrals appearing in (\ref{eq:dillen})
to obtain a point $\varphi(u,v)$ in the plane whose distance to
the curve $\alpha$ is
\[
d(u,v)=\cot\theta\left(\int^{\psi(u)}
\frac{d\sigma}{\rho(\sigma)}\right),\]
so that $f(\varphi(u,v))=\psi(u)$.
\end{remark}

\section{Minimal constant angle hypersurfaces}\label{sec:minimas}

Let us recall that a function in Euclidean space is called eikonal if its
gradient has constant length.

\begin{lemma}
\label{lemma:armonic&eikonal}
Let $f : U \subset \mathbb{R}^n \longrightarrow \mathbb{R}$ be a smooth function
defined on the connected open subset $U$. If $f$ is a non constant harmonic and eikonal function
then $f$ is linear.
\end{lemma}
\begin{proof}
The idea is to prove that $f$ is locally linear and then to use that $U$ is connected.
So in our argument we can take smaller open neighbourhoods if it were necessary.
Without loss of generality we can assume that $|\nabla f|^2=1$.
Then the level hypersurfaces of $f$ are equidistant embedded hypersurfaces in $\mathbb{R}^n$ because
the distance between two level hypersurfaces is measured along the integral curves of the vector field 
$\nabla f$, which has constant length.
Since $f$ is harmonic and eikonal every level hypersurface $f^{-1}(t)$ of $f$ is minimal in $\mathbb{R}^n$
because the mean curvature vector field $H$ of the level hypersurfaces is given by
\begin{equation}
\label{eqn:meancurv-levelhyper}
H= -\frac{1}{| \nabla f|}\triangle f + \frac{1}{|\nabla f|^2} \nabla |\nabla f| ,
\end{equation}
see \cite{MR934020} for details.
As we said before, in our case we can conclude that $H \equiv 0$, i.e. every level hypersurface is minimal.
So, $\{ f^{-1}(t) \}_{t \in f(U)}$ is a family of equidistant minimal hypersurfaces of $\mathbb{R}^n$.  
We will prove that this is possible if and only if every level hypersurface in the family
is a hyperplane.\\
Let $\lambda_1, \lambda_2, \ldots \lambda_{n-1}$ be the principal curvatures of $f^{-1}(t_0)$.
It is known that for every $t \in F(U)$ close to $t_0$, the principal curvatures of $f^{-1}(t)$ are given by
$$ \frac{\lambda_1}{1-(t-t_0)\lambda_1}, \frac{\lambda_2}{1-(t-t_0)\lambda_2},  \ldots , 
\frac{\lambda_{n-1}}{1-(t-t_0)\lambda_{n-1}} . $$
This is a consequence of the relation between the shape operator $A$ of $f^{-1}(t_0)$
and the shape operator $A_t$ of $f^{-1}(t)$: $A_t = (I-tA)^{-1} A$. See \cite{MR1075013} page 38.\\
Since every level hypersurface $f^{-1}(t)$ of $f$ is minimal, the mean curvature of $f^{-1}(t)$ is zero: 
$$ \frac{\lambda_1}{1-(t-t_0)\lambda_1} + \frac{\lambda_2}{1-(t-t_0)\lambda_2} +  \ldots + 
\frac{\lambda_{n-1}}{1-(t-t_0)\lambda_{n-1}} =0. $$

Taking the derivative with respect to $t$ and evaluating in $t=t_0$ we obtain that 
$ \lambda_1^2 + \lambda_2^2 +  \ldots + \lambda_{n-1}^2 =0$, which implies that 
$ \lambda_1 = \lambda_2 =  \ldots = \lambda_{n-1} =0$. Therefore $f^{-1}(t)$ is totally geodesic, i.e.
it is part of a hyperplane. This proves that $f$ is linear.
\end{proof}

The next Corollary \ref{coro:minima-ca}, improves Theorem 2.8 in \cite{MR2520724} which says that
a constant angle hypersurface $M$ in Euclidean space is minimal if and only if every slice of $M$
is also minimal. Our Corollary here gives a complete, explicit classification of these hypersurfaces.

\begin{corollary}
\label{coro:minima-ca}
Let $M$ be a connected constant angle hypersurface in $\mathbb{R}^n$ with respect to a constant direction $X$.
If $M$ is minimal then  either $M$ is part of a cylinder, over a minimal hypersurface in $\mathbb{R}^{n-1}$ 
or $M$ is part of a hyperplane.
\end{corollary}
\begin{proof}
We can assume that $X$ is a unit vector field.
If $X$ is tangent to $M$, then it is clear that $M$ is part of a cylinder over a hypersurface $L$ in a $\mathbb{R}^{n-1}$
orthogonal to $X$. Moreover, $L$ should be minimal because $M$ is minimal.\\
If $X$ is transversal to $M$ then $M$ if the graph of a smooth function $f$, the height function in direction $X$.
Since $M$ is minimal, every slice of $M$ with hyperplanes orthogonal to $X$ is minimal in the Euclidean ambient,
which follows from Theorem 2.8 of \cite{MR2520724}.
Equivalently, every level hypersurface of $f$ is minimal. Under the hypothesis that $f$ is eikonal and using
relation (\ref{eqn:meancurv-levelhyper}), the latter condition holds if and only if $f$ is a harmonic function.
So, $f$ is an eikonal and harmonic function. By Lemma \ref{lemma:armonic&eikonal}, $f$ is linear.
Therefore, $M$ is part of a hyperplane.
\end{proof}

\section*{Acknowledgements}

The second named author thanks the hospitality of Universidad Au\-t\'o\-no\-ma de Yucat\'an during the preparation of this paper. 
The third named author wants to thank Antonio J. Di~Scala for many useful conversations at the
Politecnico di Torino about transnormal and eikonal functions, as
well as for his suggestion of some references in such topics.

\bibliographystyle{plain}
\bibliography{references}

\noindent {\bf Authors' Addresses:}

\noindent Eugenio Garnica, Departamento de Matem\'aticas, \\
Facultad de Ciencias, UNAM, 04510 DF, M\'exico\\
 E-mail: garnica@servidor.unam.mx\\

\noindent Oscar Palmas, Departamento de Matem\'aticas,\\ 
Facultad de Ciencias, UNAM, 04510 DF, M\'exico\\
E-mail address: oscar.palmas@ciencias.unam.mx\\

\noindent Gabriel Ruiz-Hern\'andez, Instituto de Matem\'aticas,\\
UNAM, 04510 DF, M\'exico\\
E-mail: gruiz@matem.unam.mx

\end{document}